\documentclass{amsart}
\usepackage{amssymb,amsfonts}

\newcommand{\IN}{\mathbb N}
\newcommand{\IZ}{\mathbb Z}

\newcommand{\w}{\omega}
\newcommand{\F}{\mathcal F}

\newcommand{\A}{\mathcal A}
\newcommand{\rank}{\mathrm{rank}}
\newcommand{\Ra}{\Rightarrow}
\newcommand{\Tr}{\mathsf{Tr}}
\newcommand{\WF}{\mathsf{WF}}
\newcommand{\U}{\mathcal{U}}

\title{On thin-complete ideals of subsets of groups}
\author{Taras Banakh and Nadya Lyaskovska}
\address{Faculty of Mechanics and Mathematics, Ivan Franko National University of Lviv, Universytetska 1, 79000, Ukraine}
\email{tbanakh@yahoo.com; lyaskovska@yahoo.com}

\newtheorem{theorem}{Theorem}[section]
\newtheorem{lemma}[theorem]{Lemma}
\newtheorem{corollary}[theorem]{Corollary}
\newtheorem{proposition}[theorem]{Proposition}
\theoremstyle{definition}
\newtheorem{remark}[theorem]{Remark}
\newtheorem{example}[theorem]{Example}
\newtheorem{definition}[theorem]{Definition}
\newtheorem{claim}[theorem]{Claim}
\newtheorem{question}[theorem]{Question}

\begin{document}
\begin{abstract} Let $\F\subset\mathcal P_G$ be a left-invariant lower family of subsets of a group $G$. A subset $A\subset G$ is called {\em $\F$-thin} if $xA\cap yA\in\F$ for any distinct elements $x,y\in G$. The family of all $\F$-thin subsets of $G$ is denoted by $\tau(\F)$. If $\tau(\F)=\F$, then $\F$ is called {\em thin-complete}. The {\em thin-completion} $\tau^*(\F)$ of $\F$ is the smallest thin-complete subfamily of $\mathcal P_G$ that contains $\F$. 

Answering questions of Lutsenko and Protasov, we prove that a set $A\subset G$ belongs to $\tau^*(G)$ if and only if for any sequence $(g_n)_{n\in\w}$ of non-zero elements of $G$ there is $n\in\w$ such that $$\bigcap\limits_{i_0,\dots,i_n\in\{0,1\}}g_0^{i_0}\cdots g_n^{i_n}A\in\F.$$
Also we prove that for an additive family $\F\subset\mathcal P_G$ its thin-completion $\tau^*(\F)$ is additive. If the group $G$ is countable and torsion-free, then the completion $\tau^*(\F_G)$ of the ideal $\F_G$ of finite subsets of $G$ is coanalytic and not Borel in the power-set $\mathcal P_G$ endowed with the natural compact metrizable topology. 
\end{abstract}

\maketitle
\section{Introduction}

This paper was motivated by problems posed by Ie.~Lutsenko and I.V.~Protasov in a preliminary version of the paper \cite{LP2} devoted to relatively thin sets in groups. 

Following \cite{LP1}, we say that a subset $A$ of a group $G$ is {\em thin} if for any distinct points $x,y\in G$ the intersection $xA\cap yA$ is finite. In \cite{LP2} (following the approach of \cite{BL}) Lutsenko and Protasov generalized the notion of a thin set to that of $\F$-thin set where $\F$ is a family of subsets of $G$. By $\mathcal P_G$ we shall denote the Boolean algebra of all subsets of the group $G$.

We shall say that a family $\F\subset\mathcal P_G$ is 
\begin{itemize}
\item {\em left-invariant} if $xF\in\F$ for all $F\in\F$ and $x\in G$, and  
\item {\em additive} if $A\cup B\in\F$ for all $A,B\in\F$;
\item {\em lower} if $A\in\F$ for any $A\subset B\in\F$;
\item {\em an ideal} if $\F$ is lower and additive. 
\end{itemize}

Let $\F\subset\mathcal P_G$ be a left-invariant lower family of subsets of a group $G$. A subset $A\subset G$ is defined to be {\em $\F$-thin} if for any distinct points $x,y\in G$ we get $xA\cap yA\in\F$. The family of all $\F$-thin subsets of $G$ will be denoted by $\tau(\F)$. It is clear that $\tau(\F)$ is a left-invariant lower family of subsets of $G$ and $\F\subset\tau(\F)$. If $\tau(\F)=\F$, then the family $\F$ will be called {\em thin-complete}. 
\smallskip

Let $\tau^*(\F)$ be the intersection of all thin-complete families $\tilde \F\subset\mathcal P_G$ that contain $\F$. It is clear that $\tau^*(\F)$ is the smallest thin-complete family containing $\F$. This family is called the {\em thin-completion} of $\F$.

The family $\tau^*(\F)$ has an interesting hierarchic structure that can be described as follows. Let $\tau^0(\F)=\F$ and for each ordinal $\alpha$ put $\tau^\alpha(\F)$ be the family of all sets $A\subset G$ such that for any distinct points $x,y\in G$ we get $xA\cap yA\in \bigcup_{\beta<\alpha}\tau^{\beta}(\F)$. So, $$\tau^\alpha(\F)=\tau(\tau^{<\alpha}(\F))\mbox{ \ where \ }\tau^{<\alpha}(\F)=\bigcup_{\beta<\alpha}\tau^\beta(\F).$$ 
By Proposition 3 of \cite{LP2}, $\tau^*(\F)=\bigcup\limits_{\alpha<|G|^+}\tau^\alpha(\F)$.  

The following theorem (that will be proved in Section~\ref{s2}) answers the problem of combinatorial characterization of the family $\tau^*(\F)$ posed by Ie.~Lutsenko and I.V.~Protasov.
Below by $e$ we denote the neutral element of the group $G$.

\begin{theorem}\label{char-t} Let $\F\subset\mathcal P_G$ be a left-invariant lower family of subsets of a group $G$. A subset $A\subset G$ belongs to the family $\tau^*(\F)$ if and only if for any sequence $(g_n)_{n\in\w}\in (G\setminus\{e\})^\IN$ there is a number $n\in\w$ such that 
$$\bigcap_{k_0,\dots,k_{n}\in\{0,1\}}g_0^{k_0}\cdots g_n^{k_n}A\in\F.$$
\end{theorem}

We recall that a family $\F\subset\mathcal P_G$ is called {\em additive} if $\{A\cup B:A,B\in\F\}\subset\F$. 
It is clear that the family $\F_G$ of finite subsets of a group $G$ is additive. If $G$ is an infinite Boolean group, then the family $\tau^*(\F_G)=\tau(\F_G)$ is not additive, see Remark 3 in \cite{LP2}. For torsion-free groups the situation is totally diferent. Let us recall that a group $G$ is {\em torsion-free} if each non-zero element of $G$ has infinite order.

\begin{theorem}\label{ideal} For a torsion-free group $G$ and a left-invariant ideal $\F\subset\mathcal P_G$ the family $\tau^{<\alpha}(\F)$ is additive for any limit ordinal $\alpha$. In particular, the thin-completion $\tau^*(\F)$ of $\F$ is an ideal in $\mathcal P_G$.
\end{theorem}

We define a subset of a group $G$ to be {\em $*$-thin} if its belongs to the thin-completion $\tau^*(\F_G)$ of the family $\F_G$ of all finite subsets of the group $G$. By Proposition 3 of \cite{LP2}, for each countable group $G$ we get  $\tau^*(\F_G)=\tau^{<\w_1}(\F_G)$. It is natural to ask if the equality $\tau^*(\F_G)=\tau^{<\alpha}(\F_G)$ can happen for some cardinal $\alpha<\w_1$. If the group $G$ is Boolean, then the answer is affirmative: $\tau^*(\F)=\tau^1(\F)$ according to Theorem 1 of \cite{LP2}. The situation is different for non-torsion groups: 

\begin{theorem}\label{m2} If an infinite group $G$ contains an abelian torsion-free subgroup $H$ of cardinality $|H|=|G|$, then $\tau^*(\F_G)\ne\tau^{\alpha}(\F_G)\ne\tau^{<\alpha}(\F_G)$ for each ordinal $\alpha<|G|^+$.
\end{theorem}

Theorems~\ref{ideal} and \ref{m2} will be proved in Sections~\ref{s:ideal} and \ref{s5}, respectively. 
In Section~\ref{s6} we shall study the Borel complexity of the family $\tau^*(\F_G)$ for a countable group $G$. In this case the power-set $\mathcal P_G$ carries a natural compact metrizable topology, so we can talk about topological properties of subsets of $\mathcal P_G$. 

\begin{theorem} For a countable group $G$ and a countable ordinal $\alpha$ the subset $\tau^\alpha(\F_G)$ of $\mathcal P_G$ is Borel while $\tau^*(\F_G)=\tau^{<\w_1}(\F_G)$ is coanalytic. If $G$ contains an element of infinite order, then the space $\tau^*(\F_G)$ is coanalytic but not analytic.
\end{theorem}

\section{Preliminaries on well-founded posets and trees}

In this section we collect the neccessary information on well-founded posets and trees. A {\em poset} is an abbreviation from a {\em partially ordered set}. A poset $(X,\le)$ is {\em well-founded} if each subset $A\subset X$ has a maximal element $a\in A$ (this means that each element $x\in A$ with $x\ge a$ is equal to $a$). In a well-founded poset $(X,\le)$ to each point $x\in X$ we can assign the ordinal $\rank_X(x)$ defined by the recursive formula:
$$\rank_X(x)=\sup\{\rank_X(y)+1:y>x\}$$where $\sup\emptyset=0$. Thus maximal elements of $X$ have rank 0, their immediate predecessors 1, and so on. If $X$ is not empty, then the ordinal $\rank(X)=\sup\{\rank_X(x)+1:x\in X\}$ is called the {\em rank} of the poset $X$. In particular, a one-element poset has rank 1. If $X$ is empty, then we put $\rank(X)=0$. \smallskip

A {\em tree}  is a poset $(T,\le)$ with the smallest element $\emptyset_T$ such that for each $t\in T$ the lower set ${\downarrow}t=\{s\in T:s\le t\}$ is well-ordered in the sense that each subset $A\subset{\downarrow}t$ has the smallest element.
A {\em branch} of a tree $T$ is any maximal linearly ordered subset of $T$.
If a tree is well-founded, then all its branches are finite.

 A subset $S\subset T$ of a tree is called a {\em subtree} if it is a tree with respect to the induced partial order. A subtree $S\subset T$ is {\em lower} if $S={\downarrow}S=\{t\in T:\exists s\in S\;\;t\le s\}$.

All trees that appear in this paper are (lower) subtrees of the tree $X^{<\w}=\bigcup_{n\in\w}X^n$ of finite sequences of a set $X$. The tree $X^{<\w}$ carries the following partial order:
 $$(x_0,\dots,x_n)\le (y_0,\dots,y_m) \mbox{ iff $n\le m$ and $x_i=y_i$ for all $i\le n$.}$$ 
The empty sequence $s_\emptyset\in X^0$ is the smallest element (the root) of the tree $X^{<\w}$. For a finite sequence $s=(x_0,\dots,x_n)\in X^{<\w}$ and an element $x\in X$ by $s\hat{\;}x=(x_0,\dots,x_n,x)$ we denote the concatenation of $s$ and $x$. So, $s\hat{\;}x$ is one of $|X|$ many immediate successors of $s$.  The set of all branches of $X^{<\w}$ can be naturally identified with the countable power $X^\w$. For each branch $s=(s_n)_{n\in\w}\in X^\w$ and $n\in\w$ by $s|n=(s_0,\dots,s_{n-1})$ we denote the initial interval of length $n$. 

Let $\Tr\subset \mathcal P_{X^{<\w}}$ denote the family of all lower subtrees of the tree $X^{<\w}$ and $\WF\subset\Tr$ be the subset consisting of all well-founded lower subtrees of $X^{<\w}$.

In Section~\ref{s6} we shall exploit some deep facts about the descriptive properties of the sets $\WF\subset\Tr\subset\mathcal P_{X^{<\w}}$ for a  countable set $X$. In this case the tree $X^{<\w}$ is countable and the power-set $\mathcal P_{X^{<\w}}$ carries a natural compact metrizable topology of the Tychonov power $2^{X^{<\w}}$. 
So, we can speak about topological properties of the subsets $\WF$ and $\Tr$ of the compact metrizable space $\mathcal P_{X^{<\w}}$.

We recall that a topological space $X$ is {\em Polish} if $X$ is homeomorphic to a separable complete metric space. A subset $A$ of a Polish space $X$ is called
\begin{itemize}
\item {\em Borel} if $A$ belongs to the smallest $\sigma$-algebra that contains all open subsets of $X$;
\item {\em analytic} if $A$ is the image of a Polish space $P$ under a continuous map $f:P\to A$;
\item {\em coanalytic} if $X\setminus A$ is analytic.
\end{itemize}
By Souslin's Theorem 14.11 \cite{Ke}, a subset of a Polish space is Borel if and only if it is both analytic and coanalytic. By $\Sigma^1_1$ and $\Pi^1_1$ we denote the classes of spaces homeomorphic to analytic and coanalytic subsets of Polish spaces, respectively.

A coanalytic subset $X$ of a compact metric space $K$ is called {\em $\Pi^1_1$-complete} if  for each coanalytic subset $C$ of the Cantor cube $2^\w$ there is a continuous map $f:2^\w\to K$ such that $f^{-1}(X)=C$. It follows from the existence of a coanalytic non-Borel set in $2^\w$ that each $\Pi^1_1$-complete set $X\subset K$ is non-Borel.  

The following deep theorem is classical and belongs to Lusin, see \cite[32.B and 35.23]{Ke}.

\begin{theorem}\label{WF} Let $X$ be a countable set. 
\begin{enumerate}
\item The subspace $\Tr$ is closed (and hence compact) in $\mathcal P_{X^{<\w}}$.
\item The set of well-founded trees $\WF$ is $\Pi^1_1$-complete in $\Tr$. In particular, $\WF$ is coanalytic but not analytic (and hence not Borel).
\item For each ordinal $\alpha<\w_1$ the subset $\WF_\alpha=\{T\in\WF:\rank(T)\le\alpha\}$ is Borel in $\Tr$.
\item Each analytic subspace of $\WF$ lies in $\WF_\alpha$ for some ordinal $\alpha<\w_1$.
\end{enumerate}
\end{theorem}

\section{Combinatorial characterization of $*$-thin subsets}\label{s2}

In this section we prove Theorem~\ref{char-t}. Let $\F\subset\mathcal P_G$ be a left-invariant lower family of subsets of a group $G$. Theorem~\ref{char-t} trivially holds if $\F=\mathcal P_G$ (which happens if and only if $G\in\F$). So, it remains to consider the case $G\notin\F$.
Let $e$ be the neutral element of $G$ and $G_\circ=G\setminus\{e\}$.
We shall work with the tree $G_\circ^{<\w}$ discussed in the preceding section.

Let $A$ be any subset of $G$. To each finite sequence $s\in G_\circ^{<\w}$ assign the set $A_s\subset G$, defined by induction: $A_\emptyset=A$ and  $A_{s\hat{\,}x}=A_s\cap xA_s$ and for $s\in G_\circ^{<\w}$ and $x\in G_\circ$. Repeating the inductive argument of the proof of Proposition 2 \cite{LP2}, we can obtaine the following direct description of the sets $A_s$:

\begin{claim} For every sequence $s=(g_0,\dots,g_n)\in G_\circ^{<\w}$ 
$$A_s=\bigcap_{k_0,\dots,k_n\in\{0,1\}}g_0^{k_0}\cdots g_n^{k_n}A.$$
\end{claim}

The set $$T_A=\{s\in G_\circ^{<\w}:A_s\notin\F\}$$ is a subtree of $G_\circ^{<\w}$ called the {\em $\tau$-tree} of the set $A$. 

For a non-zero ordinal $\alpha$ let $-1+\alpha$ be a unique ordinal $\beta$ such that $1+\beta=\alpha$. For $\alpha=0$ we put $-1+\alpha=-1$. It follows that $-1+\alpha=\alpha$ for each infinite ordinal $\alpha$.

\begin{theorem}\label{t1.2} A set $A\subset G$ belongs to the family $\tau^{\alpha}(\F)$
for some ordinal $\alpha$ if and only if its $\tau$-tree $T_A$ is well-founded and has $\rank(T_A)\le-1+\alpha+1$.
\end{theorem}

\begin{proof} By induction on $\alpha$. Observe that $A\in\tau^{0}(\F)=\F$ if and only if $T_A=\emptyset$ if and only if $\rank(T_A)=0=1+0+1$. So, Theorem holds for $\alpha=0$. 

Assume that for some ordinal $\alpha>0$ and any ordinal $\beta<\alpha$ we know that a set $A\subset G$ belongs to $\tau^{\beta}(G)$ if and only if $T_A$ is a well-founded tree with $\rank(T_A)\le-1+\beta+1$. Given a subset $A\subset G$ we should check that  that $A\in\tau^{\alpha}(\F)$ if and only if its $\tau$-tree $T_A$ is well-founded and has $\rank(T_A)\le-1+\alpha+1$.

First assume that $A\in\tau^{\alpha}(\F)$. Then for every $x\in G_\circ$ the set $A\cap xA$ belongs to $\tau^{\beta_x}(\F)\subset\tau^{<\alpha}(\F)$ for some ordinal $\beta_x<\alpha$. By the inductive assumption, the $\tau$-tree $T_{A\cap xA}$ is well-founded and has $\rank(T_{A\cap xA})\le-1+\beta_x+1$.

If $A\in\tau(\F)$, then $T_A\subset\{s_\emptyset\}$ and $\rank(T_A)\le 1\le-1+\alpha+1$. So, we can assume that $A\notin\tau(\F)$. In this case each point $x\in G_\circ=G_\circ^1$ considered as the sequence $(x)\in G^1$ of length 1 belongs to the $\tau$-tree $T_A$ of the set $A$. So we can consider the upper set $T_A(x)=\{s\in T_A:s\ge x\}$ and observe that the subtree $T_A(x)$ of $T_A$ is isomorphic to the $\tau$-tree $T_{A\cap xA}$ of the set $A\cap xA$ and hence $\rank(T_A(x))=\rank(T_{A\cap xA})\le-1+\beta_x+1$.
It follows that $$
\begin{aligned}
\rank(T_A)&=\rank_{T_A}(s_\emptyset)+1=\big(\sup_{x\in G_\circ}(\rank_{T_A}(x)+1)\big)+1=\\
&=\big(\sup_{x\in G_\circ}\rank\,T_A(x)\big)+1\le\big(\sup_{x\in G_\circ}(-1+\beta_x+1\big)\big)+1\le -1+\alpha+1.
\end{aligned}
$$
\smallskip

Now assume conversely that the $\tau$-tree $T_A$ of $A$ is well-founded and has $\rank(T_A)\le-1+\alpha+1$. For each $x\in G_\circ$, find a unique ordinal $\beta_x$ such that $-1+\beta_x=\rank_{T_A}(x)$. It follows from 
$$-1+\beta_x+2=\rank_{T_A}(x)+2\le\rank_{T_A}(s_\emptyset)+1=\rank(T_A)\le-1+\alpha+1$$that $\beta_x<\alpha$. Since the subtree $T_A(x)=T_A\cap{\uparrow}x$ is isomorphic to the $\tau$-tree $T_{A\cap xA}$ of the set $A\cap xA$, we conclude that $T_{A\cap xA}$ is well-founded and has $\rank(T_{A\cap xA})=\rank(T_A(x))=\rank_{T_A}(x)+1=-1+\beta_x+1$. Then the inductive assumption guarantees that $A\cap xA\in\tau^{\beta_x}(\F)\subset\tau^{<\alpha}(\F)$ and hence  $A\in\tau^{\alpha}(\F)$ by the definition of the family $\tau^\alpha(\F)$.
\end{proof}  

As a corollary of Theorem~\ref{t1.2}, we obtain the following characterization proved in \cite{LP2}:

\begin{corollary}\label{tau-n} A subset $A\subset G$ belongs to the family $\tau^n(\F)$ for some $n\in\w$ if and only if for each sequence $(g_i)_{i=0}^n\in G_\circ^{n+1}$ we get $$\bigcap\limits_{k_0,\dots,k_n\in\{0,1\}}g_0^{k_0}\cdots g_n^{k_n}A\in\F.$$
\end{corollary}

Theorem~\ref{t1.2} also implies the following explicit description of the family $\tau^*(\F)$, which was announced in Theorem~\ref{char-t}:

\begin{corollary}\label{c3.4} For a subset $A\subset G$ the following conditions are equivalent:
\begin{enumerate}
\item $A\in\tau^*(\F)$;
\item the $\tau$-tree $T_A$ of $A$ is well-founded;
\item for each sequence $(g_n)_{n\in\w}\in G_\circ^\w$ there is $n\in\w$ such that $(g_0,\dots,g_n)\notin T_A$;
\item  for each sequence $(g_n)_{n\in\w}\in G_\circ^\w$ there is $n\in\w$ such that $$\bigcap\limits_{k_0,\dots,k_n\in\{0,1\}}g_0^{k_0}\cdots g_n^{k_n}A\in\F.$$
\end{enumerate}
\end{corollary}

\section{The additivity of the families $\tau^{<\alpha}(\F)$}\label{s:ideal}

In this section we shall prove Theorem~\ref{ideal}. 
Let $G$ be an infinite group and $e$ be the neutral element of $G$. 

 For a natural number $m$ let $2^m$ denote the finite cube $\{0,1\}^m$. For vectors $\mathbf g=(g_1,\dots,g_m)\in (G\setminus\{e\})^m$ and $\mathbf x=(x_1,\dots,x_m)\in 2^m$ let $$\mathbf g^{\mathbf x}=g_1^{x_1}\cdots g_m^{x_m}\in G.$$ 

A function $f:2^m\to G$ to a group $G$ will be called {\em cubic} if there is a vector $\mathbf g=(g_1,\dots,g_m)\in (G\setminus\{e\})^m$ such that $f(x)=\mathbf g^x$ for all $x\in 2^m$.   

\begin{lemma}\label{alpha} If the group $G$ is torsion-free, then for every $n\in\IN$, $m>(n-1)^2$, and a cubic function $f:2^{m}\to G$ we get $|f(2^{m})|>n$.
\end{lemma}

\begin{proof} Assume conversely that $|f(2^m)|\le n$. Consider the set $B=\{(k_1,\dots,k_m)\in 2^m:\sum_{i=1}^mk_i=1\}$ having cardinality $|B|=m>(n-1)^2$. Since $e\notin f(B)$, we conclude that $|f(B)|\le |f(2^m)|-1\le n-1$ and hence $|f^{-1}(y)|\ge n$ for some $y\in f(B)$. Let $B_y=f^{-1}(y)$ and observe that $f(2^m)\supset \{e,y,y^2,\dots,y^{|B_y|}\}$ and thus $|f(2^m)|\ge |B_y|+1\ge n+1$, which contradicts our assumption.
\end{proof}

For every $n\in\IN$ let $c(n)$ be the smallest number $m\in\IN$ such that for each cubic function $f:2^m\to G$ we get $|f(2^m)|>n$. It is easy to see that $c(n)\ge n$. On the other hand,  Lemma~\ref{alpha} implies that $c(n)\le (n-1)^2+1$ if $G$ is torsion-free. 

For a family $\F$ and a natural number $n\in\IN$, let
$$\bigvee_n\F=\{\cup\A:\A\subset\F,\;|\A|\le n\}.$$

\begin{lemma}\label{sum} Let $\F\subset\mathcal P_G$ be a left-invariant lower family of subsets in a torsion-free group $G$. For every $n\in\IN$ we get
$$\bigvee_n\tau(\F)\subset \tau^{c(n)-1}(\bigvee_{m}\F)$$where $m=n^{2^{c(n)}}$.
\end{lemma}

\begin{proof} 
Fix any $A\in\bigvee\limits_n\tau(\F)$ and write it as the union $A=A_1\cup\dots\cup A_n$ of sets $A_1,\dots,A_n\in\tau(\F)$. 
The inclusion $A\in\tau^{c(n)-1}(\bigvee\limits_{m}\F)$ will follow from Corollary~\ref{tau-n} as soon as we check that 
$$\bigcap_{x\in 2^{c(n)}}\mathbf g^xA\in\bigvee_m\F$$for each vector $\mathbf g\in (G\setminus\{e\})^{c(n)}$. De Morgan's law guarantees that
$$
\bigcap_{x\in 2^{c(n)}}\mathbf g^x\cdot (\bigcup_{i=1}^nA_i)=
\bigcup_{f\in n^{2^{c(n)}}}\bigcap_{x\in 2^{c(n)}}\mathbf g^xA_{f(x)}.$$
So, the proof will be complete as soon as we check that for every function $f:2^{c(n)}\to n$ the set $\bigcap\limits_{x\in 2^{c(n)}}\mathbf  g^xA_{f(x)}$ belongs to $\F$. The vector $\mathbf g\in (G\setminus\{e\})^{c(n)}$ induces the cubic function $g:2^{c(n)}\to G$, $g:x\mapsto\mathbf g^x$. The definition of the function $c(n)$ guarantees that $|g(2^{c(n)})|>n$. The function $f:2^{c(n)}\to n$ can be thought as a coloring of the cube $2^{c(n)}$ into $n$ colors. Since $|g(2^{c(n)})|>n$, there are  two points $y,z\in 2^{c(n)}$ colored by the same color such that $g(y)\ne g(z)$. Then $\mathbf g^y=g(y)\ne g(z)=\mathbf g^z$ but $f(y)=f(z)=k$ for some $k\le n$.
Consequently,  $$\bigcap_{x\in 2^{c(n)}}\mathbf g^xA_{f(x)}\subset \mathbf g^y A_k\cap \mathbf g^z A_k\in\F$$because the set $A_k\in\tau(\F)$.
\end{proof}

Now consider the function $c:\IN\times\w\to\w$ defined recursively as $c(n,0)=0$ for all $n\in\IN$ and $c(n,k+1)=c(n)-1+c(n^{2^{c(n)}},k)$ for $(n,k)\in\IN\times\w$. Observe that $c(n,1)=c(n)-1$ for all $n\in\IN$.

\begin{lemma}\label{l4.3} If the group $G$ is torsion-free and  $\F\subset\mathcal P_G$ is a left-invariant ideal, then 
$$\bigvee_n\tau^k(\F)\subset\tau^{c(n,k)}(\F)$$for all pairs $(n,k)\in\IN\times\w$.
\end{lemma}

\begin{proof} By induction on $k$. For $k=0$ the equality  $\bigvee_n\tau^0(\F)=\F=\tau^{c(n,0)}(\F)$ holds because $\F$ is additive.

Assume that Lemma is true for some $k\in\w$. By Lemma~\ref{sum} and by the inductive assumption, for every $n\in\IN$ we get
$$
\begin{aligned}
\bigvee_n\tau^{k+1}(\F)&=\bigvee_n\tau(\tau^k(\F))\subset\tau^{c(n)-1}\big(\bigvee_{n^{2^{c(n)}}}\tau^k(\F)\big)\subset \\
&\tau^{c(n)-1}(\tau^{c(n^{2^{c(n)}},k)}(\F))=
\tau^{c(n)-1+c(n^{2^{c(n)}},k)}(\F)=
\tau^{c(n,k+1)}(\F).
\end{aligned}$$
\end{proof}

Now we are able to present:

\begin{proof}[Proof of Theorem~\ref{ideal}] Assume that $G$ is a torsion-free group $G$ and $\F\subset\mathcal P_G$ is a left-invariant ideal. By transfinite induction we shall prove that for each limit ordinal $\alpha$ the family $\tau^{<\alpha}(\F)$ is additive. 
For the smallest limit ordinal $\alpha=0$ the additivity of the family $\tau^0(\F)=\F$ is included into the hypothesis. Assume that for some non-zero limit ordinal $\alpha$ we have proved that the families $\tau^{<\beta}(\F)$ are additive for all limit ordinals $\beta<\alpha$. Two cases are possible:

1) $\alpha=\beta+\w$ for some limit ordinal $\beta$. By the inductive assumption, the family $\tau^{<\beta}(\F)$ is additive. Then Lemma~\ref{l4.3} implies that the family $\tau^{<\alpha}(\F)=\tau^{<\w}(\tau^{<\beta}(\F))$ is additive.

2) $\alpha=\sup B$ for some family $B\not\ni \alpha$ of limit ordinals. By the inductive assumption for each limit ordinal $\beta\in B$ the family $\tau^{<\beta}(\F)$ is additive and then the union $$\tau^{<\alpha}(\F)=\bigcup_{\beta\in B}\tau^{<\beta}(\F)$$ is additive too. 
 
This completes the proof of the additivity of the families $\tau^{<\alpha}(\F)$ for all limit ordinals $\alpha$. Since the torsion-free group $G$ is infinite, the ordinal $\alpha=|G|^+$ is limit and hence the family $\tau^*(\F)=\tau^{<\alpha}(\F)$ is additive. Being left-invariant and lower, the family $\tau^*(\F)$ is a left-invariant ideal in $\mathcal P_G$.
\end{proof}

\begin{remark} Theorem~\ref{ideal} is not true for an infinite Boolean group $G$. In this case Theorem 1(2) of \cite{LP2} implies that $\tau^*(\F_G)=\tau(\F_G)$. Then for any infinite thin subset $A\subset G$ and any $x\in G\setminus\{e\}$ the union $A\cup xA$ is not thin as $(A\cup xA)\cap x(A\cup xA)=A\cup xA$ is infinite. Consequently, the family $\tau^*(\F_G)=\tau(\F_G)$ is not additive.
\end{remark}

\section{$h$-Invariant families of subsets in groups}\label{s:hinv}

 Let $G$ be a group and $h:H\to K$ be an isomorphism between subgroups of $G$. A family $\F$ of subsets of $G$ is called {\em $h$-invariant} if a subset $A\subset H$ belongs to $\F$ if and only if $h(A)\in\F$. 

\begin{example} The ideal $\F_\IZ$ of finite subsets of the group $\IZ$ is $h$-invariant for each isomorphism $h_k:\IZ\to k\IZ$, $h:x\mapsto kx$, where $k\in\IN$.
\end{example}

\begin{proposition}\label{hinv} Let $h:H\to K$ be an isomorphism between subgroups of a group $G$. For any $h$-invariant family $\F\subset\mathcal P_G$ and any ordinal $\alpha$ the family $\tau^\alpha(\F)$ is $h$-invariant.
\end{proposition}

\begin{proof} For $\alpha=0$ the $h$-invariance of $\tau^0(\F)=\F$ follows from our assumption. Assume that for some ordinal $\alpha$ we have established that the families $\tau^\beta(\F)$ are $h$-invariant for all ordinals $\beta<\alpha$. Then the union $\tau^{<\alpha}(\F)=\bigcup_{\beta<\alpha}\tau^\beta(\F)$ is also $h$-invariant. 

We shall prove that the family $\tau^\alpha(\F)$ is $h$-invariant. 
Given a set $A\subset H$ we need to prove that $A\in\tau^\alpha(\F)$ if and only if $h(A)\in\tau^\alpha(\F)$. 

Assume first that $A\in\tau^\alpha(\F)$. To show that $h(A)\in\tau^\alpha(\F)$, take any element $y\in G\setminus\{e\}$. If $y\notin K$, then $h(A)\cap yh(A)=\emptyset\in \tau^{<\alpha}(\F)$. If $y\in K$, then $y=h(x)$ for some $x\in H$ and then $h(A)\cap yh(A)=h(A\cap xA)\in\tau^{<\alpha}(\F)$ since $A\cap xA\in\tau^{<\alpha}(\F)$ and the family $\tau^{<\alpha}(\F)$ is $h$-invariant.

Now assume that $A\notin\tau^\alpha(\F)$. Then there is an element $x\in G\setminus\{e\}$ such that $A\cap xA\notin\tau^{<\alpha}(\F)$.
Since $A\subset H$, the element $x$ must belong to $H$ (otherwise $A\cap xA=\emptyset\in\tau^{<\alpha}(\F)$).
 Then for the element $y=h(x)$ we get $h(A)\cap yh(A)\notin\tau^{<\alpha}(\F)$ by the $h$-invariance of the family $\tau^{<\alpha}(\F)$. Consequently, $h(A)\notin\tau^\alpha(\F)$.
\end{proof}   
   
\begin{corollary} Let $h:H\to K$ be an isomorphism between subgroups of a group $G$. For any $h$-invariant family $\F\subset\mathcal P_G$  the family $\tau^*(\F)$ is $h$-invariant.
\end{corollary} 

\begin{definition} A {\em left-invariant} family $\F\subset\mathcal P_G$ of subsets of a group $G$ is called 
\begin{itemize}
\item {\em auto-invariant} if $\F$ is $h$-invariant for each injective homomorphism $h:G\to G$;
\item {\em sub-invariant} if $\F$ is $h$-invariant for each isomorphism $h:H\to K$ between subgroups $K\subset H$ of $G$.
\item {\em strongly invariant} if $\F$ is $h$-invariant for each isomorphism $h:H\to K$ between subgroups of $G$.
\end{itemize}
\end{definition}

It is clear that
$$\mbox{strongly invariant } \Ra \mbox{ sub-invariant } \Ra \mbox{ auto-invariant}
$$ 

\begin{remark} Each auto-invariant family $\mathcal F\subset\mathcal P_G$, being left-invariant is also right-invariant.
\end{remark}

Proposition~\ref{hinv} implies:

\begin{corollary} If $\F\subset \mathcal P_G$ is an auto-invariant (sub-invariant, strongly invariant) family of subsets of a group $G$, then so are the families $\tau^*(\F)$ and $\tau^\alpha(\F)$ for all ordinals $\alpha$.
\end{corollary} 

It is clear that the famly $\F_G$ of finite subsets of a group $G$ is strongly invariant. 
Now we present some natural examples of families, which are not strongly invariant. Following \cite{BM}, we call a subset $A$ of a group $G$
\begin{itemize}
\item {\em large} if there is a finite subset $F\subset G$ with $G=FA$;
\item {\em small} if for any large set $L\subset G$ the set $L\setminus A$ remains large. 
\end{itemize}
It follows that the family $\mathcal S_G$ of small subsets of $G$ is a left-invariant ideal in $\mathcal P_G$. According to \cite{BM}, a subset $A\subset G$ is small if and only if for every finite subset $F\subset G$ the complement $G\setminus FA$ is large.
We shall need the following (probably known) fact.

\begin{lemma}\label{finind} Let $H$ be a subgroup of finite index in a group $G$. A subset $A\subset H$ is small in $H$ if and only if $A$ is small in $G$.
\end{lemma}

\begin{proof} First assume that $A$ is small in $G$. To show that $A$ is small in $H$, take any large subset $L\subset H$. Since $H$ has finite index in $G$, the set $L$ is large in $G$. Since $A$ is small in $G$, the complement $L\setminus A$ is large in $G$. Consequently, there is a finite subset $F\subset G$ such that $F(L\setminus A)=G$. Then for the finite set $F_H=F\cap H$, we get $F_H(L\setminus A)=H$, which means that $L\setminus A$ is large in $H$.

Now assume that $A$ is small in $H$. To show that $A$ is small in $G$, it suffices to show that for every finite subset $F\subset G$ the complement $G\setminus FA$ is large in $G$. Observe that $(G\setminus FA)\cap H=H\setminus F_HA$ where $F_H=F\cap H$. Since $A$ is small in $H$, the set $H\setminus F_HA$ is large in $H$ and hence large in $G$ (as $H$ has finite index in $G$). Then the set $G\setminus FA\supset H\setminus F_HA$ is large in $G$ too.
\end{proof}

\begin{proposition}Let $G$ be an infinite abelian group. \begin{enumerate}
\item If $G$ is finitely generated, then the ideal $\mathcal S_G$ is strongly invariant.
\item If $G$ is infinitely generated free abelian group, then the ideal $\mathcal S_G$ is not auto-invariant.
\end{enumerate}
\end{proposition} 

\begin{proof} 1. Assume that $G$ is a finitely generated abelian group. To show that $\mathcal S_G$ is strongly invariant, fix any isomorphism $h:H\to K$ between subgroups of $G$ and let $A\subset H$ be any subset. The groups $H,K$ are isomorphic and hence have the same free rank $r_0(H)=r_0(K)$. If $r_0(H)=r_0(K)<r_0(G)$, then the subgroups $H,K$ have infinite index in $G$ and hence are small. In this case the inclusions $A\in \mathcal S_G$ and $h(A)\in\mathcal S_G$ hold and so are equivalent.

If the free ranks $r_0(H)=r_0(K)$ and $r_0(G)$ coincide, then $H$ and $K$ are subgroups of finite index in the finitely generated group $G$. By Lemma~\ref{finind}, a subset $A\subset H$ is small in $G$ if and only if $A$ is small in $H$ if and only if $h(A)$ is small in the group $h(H)=K$ if and only if $h(A)$ is small in $G$.
\smallskip

2. Now assume that $G$ is an infinitely generated free abelian group. Then $G$ is isomorphic to the direct sum $\oplus^\kappa\IZ$ of $\kappa=|G|\ge\aleph_0$ many copies of the infinite cyclic group $\IZ$. Take any subset $\lambda\subset\kappa$ with infinite complement $\kappa\setminus\lambda$ and cardinality 
$|\lambda|=|\kappa|$ and fix an isomorphism $h:G\to H$ of the group $G=\oplus^\kappa\IZ$ onto its subgroup $H=\oplus^\lambda\IZ$. The subgroup $H$ has infinite index in $G$ and hence is small in $G$. Yet $h^{-1}(H)=G$ is not small in $G$, witnessing that the ideal $\mathcal S_G$ of small subsets of $G$ is not auto-invariant.
\end{proof}

\section{Thin-completeness of the families $\tau^\alpha(\F)$}\label{s5}

In this section we shall prove that in general the families $\tau^\alpha(\F)$ are not thin-complete. 
Our principal result is the following theorem that implies Theorem~\ref{m2} announced in the Introduction.

\begin{theorem}\label{t5.1} Let $G$ be a group containing a free abelian subgroup $H$ of cardinality $|H|=|G|$.
 If $\F$ is a sub-invariant ideal of subsets of $G$ such that $\tau(\F)\cap\mathcal P_H\not\subset\F$, then 
 $\tau^*(\F)\ne\tau^\alpha(\F)\ne\tau^{<\alpha}(\F)$ for all ordinals $\alpha<|G|^+$.
\end{theorem}

We divide the proof of this theorem in a series of lemmas.

\begin{lemma}\label{l6.2} Let $h:H\to K$ be an isomorphism between subgroups of a group $G$, $\F$ be an $h$-invariant left-invariant lower  family of subsets of $G$. If a subset $A\subset H$ does not belong to $\tau^\alpha(\F)$ for some ordinal $\alpha$, then for every point $z\in G\setminus \{e\}$ the set  $h(A)\cup zh(A)\notin\tau^{\alpha+1}(\F)$.
\end{lemma}

\begin{proof} Proposition~\ref{hinv} implies that $h(A)\notin \tau^\alpha(\F)$. Since $$(h(A)\cup zh(A))\cap z^{-1}(h(A)\cup zh(A))\supset h(A)\notin\tau^\alpha(\F),$$ the set $h(A)\cup zh(A)\notin \tau^{\alpha+1}(\F)$ by the definition of $\tau^{\alpha+1}(\F)$.
\end{proof}

In the following lemma for a subgroup $K$ of a group $H$ by 
$$Z_H(K)=\{z\in H:\forall x\in K\;\;zx=xz\}$$we denote the centralizer of $K$ in $H$.

\begin{lemma}\label{union2} Let $h:H\to K$ be an isomorphism between subgroups $K\subset H$ of a group $G$ such that there is a point $z\in Z_H(K)$ with $z^2\notin K$. Let $\F\subset\mathcal P_G$ be an $h$-invariant left-invariant ideal. If a subset $A\subset H$ belongs to the family $\tau^\alpha(\F)$ for some ordinal $\alpha$, then $h(A)\cup zh(A)\in\tau^{\alpha+1}(\F)$.
\end{lemma}

\begin{proof} By induction on $\alpha$. For $\alpha=0$ and $A\in\F$ the inclusion $h(A)\cup zh(A)\in\F\subset\tau(\F)$ follows from the $h$-invariance and the additivity of $\F$.

Now assume that for some ordinal $\alpha$ we have proved that for every $\beta<\alpha$ and $A\in\mathcal P_H\cap\tau^\beta(\F)$ the set $h(A)\cup zh(A)$ belongs to $\tau^{\beta+1}(\F)$.
Given any set $A\in\mathcal P_H\cap \tau^{\alpha}(\F)$,  we need to prove that $h(A)\cup zh(A)\in\tau^{\alpha+1}(\F)$. This will follow as soon as we check that $(h(A)\cup zh(A))\cap y(h(A)\cup zh(A)\in\tau^{\alpha}(\F)$ for every $y\in G\setminus\{e\}$.

If $y\notin K\cup zK\cup z^{-1}K$, then 
$$(h(A)\cup zh(A))\cap y(h(A)\cup zh(A))\subset (K\cup zK)\cap y(K\cup zK)=\emptyset\in\tau^{\alpha+1}(\F).$$

So, it remains to consider the case $y\in K\cup zK\cup z^{-1}K\subset H$.
If $y\in K$, then $$(h(A)\cup zh(A))\cap y(h(A)\cup zh(A))=(h(A)\cap yh(A))\cup z(h(A)\cap y\,h(A)).$$
Since $y\in K$, there is an element $x\in H$ with $y=h(x)$. Since $A\in\tau^{\alpha}(\F)$, $A\cap xA\in\tau^\beta(\F)$ for some $\beta<\alpha$ and then $$(h(A)\cup zh(A))\cap y(h(A)\cup zh(A))=h(A\cap xA)\cup zh(A\cap xA)\in\tau^{\beta+1}(\F)\subset\tau^\alpha(\F)$$ by the inductive assumption. 
If $y\in zK$, then $z^2\notin K$ implies that $$(h(A)\cup zh(A))\cap y(h(A)\cup zh(A))=zh(A)\cap yh(A)\subset
 zh(A)\in\tau^\alpha(\F)$$ by the $h$-invariance and the left-invariance of the family $\tau^\alpha(\F)$, see Proposition~\ref{hinv}.

If $y\in z^{-1}K$, then by the same reason,
$$(h(A)\cup zh(A))\cap y(h(A)\cup zh(A))=h(A)\cap yzh(A)\subset h(A)\in\tau^\alpha(\F).$$
\end{proof}

Given an isomorphism $h:H\to K$ between subgroups $K\subset H$ of a group $G$, for every $n\in\IN$ define the iteration $h^n:H\to K$ of the isomorphism $h$ letting $h^1=h:H\to K$ and $h^{n+1}=h\circ h^n$ for $n\ge 1$.

The isomorphism $h:H\to K$ will be called {\em expanding} if $\bigcap_{n\in\IN}h^n(H)=\{e\}$.

\begin{example} For every integer $k\ge 2$ the isomorphism $$h_k:\IZ\to k\IZ,\;\;h_k:x\mapsto kx,$$ is expanding.
\end{example} 

\begin{lemma}\label{union} Let $h:H\to K$ be an expanding isomorphism between torsion-free subgroups $K\subset H$ of a group $G$ and $\F\subset\mathcal P_G$ be an $h$-invariant left-invariant ideal of subsets of $G$. For any limit ordinal $\alpha$ and family $\{A_n\}_{n\in\w}\subset\tau^{<\alpha}(\F)$ of subsets of the group $H$, the union $A=\bigcup_{n\in\w}h^n(A_n)$ belongs to the family $\tau^{\alpha}(\F)$.
\end{lemma}

\begin{proof} First observe that $\{h^n(A_n)\}_{n\in\w}\subset\tau^{<\alpha}(\F)$ by Proposition~\ref{hinv}. To show that $A=\bigcup_{n\in\w}h^n(A_n)\in\tau^{\alpha}(\F)$ we need to check that $A\cap xA\in\tau^{<\alpha}(\F)$ for all $x\in G\setminus\{e\}$. This is trivially true if $x\notin H$ as $A\subset H$. So, we assume that $x\in H$. By the expanding property of the isomorphism $h$, there is a number $m\in\w$ such that $x\notin h^m(H)$. Put $B=\bigcup_{n=0}^{m-1}h^n(A_n)$ and observe that  $A\cap xA\subset B\cup xB\in\tau^{<\alpha}(\F)$ as $\tau^{<\alpha}(\F)$ is additive according to Theorem~\ref{ideal}.  
\end{proof}

\begin{lemma}\label{l6.6} Assume that a left-invariant ideal $\F$ on a group $G$ is $h$-invariant for some expanding isomorphism $h:H\to K$ between torsion-free subgroups $K\subset H$ of $G$ such that $Z_H(K)\not\subset K$. If $\tau(\F)\cap\mathcal P_H\not\subset\F$, then $\tau^{\alpha}(\F)\ne\tau^{<\alpha}(\F)$ for all ordinals $\alpha<\omega_1$. 
\end{lemma}

\begin{proof} Fix any point $z\in Z_K(H)\setminus K$. Since $H$ is torsion-free, $z^2\ne e$. Since the isomorphism $h$ is expanding, $z^2\notin h^m(H)$ for some $m\in\IN$. Replacing the isomorphism $h$ by its iterate $h^m$, we lose no generality assuming that $z^2\notin h(H)=K$.  

By induction on $\alpha<\w_1$ we shall prove that $
\tau^{\alpha}(\F)\cap\mathcal P_H\ne\tau^{<\alpha}(\F)\cap\mathcal P_H.
$ 

For $\alpha=1$ the non-equality $\tau(\F)\cap\mathcal P_H\ne\tau^0(\F)\cap\mathcal P_H$ is included into the hypothesis. Assume that for some ordinal $\alpha<\w_1$ we proved that
$\tau^{\beta}(\F)\cap\mathcal P_H\ne\tau^{<\beta}(\F)\cap\mathcal P_H$ for all ordinals $\beta<\alpha$.

If $\alpha=\beta+1$ is a successor ordinal, then by the inductive assumption we can find a set $A\in\tau^{\beta}(\F)\setminus\tau^{<\beta}(\F)$ in the subgroup $H$.
By Lemmas~\ref{l6.2} and \ref{union2}, $A\cup zA\in\tau^{\beta+1}(\F)\setminus\tau^{\beta}(\F)=\tau^\alpha(\F)\setminus\tau^{<\alpha}(\F)$ and we are done.

If $\alpha$ is a limit ordinal, then we can find an increasing sequence of ordinals $(\alpha_n)_{n\in\w}$ with $\alpha=\sup_{n\in\w}\alpha_n$. By the inductive assumption, for every $n\in\w$ there is a subset $A_n\subset H$ with $A_n\in\tau^{\alpha_n+1}(\F)\setminus\tau^{\alpha_n}(\F)$. Then we can put $A=\bigcup_{n\in\w}h^n(A_n)$. By Proposition~\ref{hinv}, for every $n\in\w$, we get  $$h^n(A_n)\in\tau^{\alpha_n+1}(\F)\setminus\tau^{\alpha_n}(\F)$$ and thus $A\notin\tau^{\alpha_n}(\F)$ for all $n\in\w$, which implies that $A\notin \tau^{<\alpha}(\F)$. On the other hand, Lemma~\ref{union} guarantees that $A\in\tau^{\alpha}(\F)$.
\end{proof}

\begin{lemma}\label{l6.7} Assume that a left-invariant ideal $\F$ on a group $G$ is $h$-invariant for some isomorphism $h:H\to K$ between torsion-free subgroups $K\subset H$ of $G$ such that $z^2\notin K$ for some $z\in Z_K(H)$. Assume that for an infinite cardinal $\kappa$ there are isomorphisms $h_n:H\to H_n$, $n\in\kappa$, onto subgroups $H_n\subset H$ such that $\F$ is $h_n$-invariant and $H_n\cdot H_m\cap H_k\cdot H_l=\{e\}$ for all indices $n,m,k,l\in\kappa$ with $\{n,m\}\cap\{k,l\}=\emptyset$. 

 If   $\tau(\F)\cap\mathcal P_H\not\subset\F$, then $\tau^{\alpha}(\F)\ne\tau^{<\alpha}(\F)$ for all ordinals $\alpha<\kappa^+$. 
\end{lemma}

\begin{proof}  By induction on $\alpha<\kappa^+$ we shall prove that $
\tau^{\alpha}(\F)\cap\mathcal P_H\ne\tau^{<\alpha}(\F)\cap\mathcal P_H.
$ 

For $\alpha=1$ the non-equality $\tau^1(\F)\cap\mathcal P_H\ne\tau^0(\F)\cap\mathcal P_H$ is included into the hypothesis. Assume that for some ordinal $\alpha<\kappa^+$ we proved that
$\tau^{\beta}(\F)\cap\mathcal P_H\ne\tau^{<\beta}(\F)\cap\mathcal P_H$ for all ordinals $\beta<\alpha$.

If $\alpha=\beta+1$ is a successor ordinal, then by the inductive assumption we can find a set $A\in\tau^{\beta}(\F)\setminus\tau^{<\beta}(\F)$ in the subgroup $H$.
By Lemmas~\ref{l6.2} and \ref{union2}, $h(A)\cup zh(A)\in\tau^{\beta+1}(\F)\setminus\tau^{\beta}(\F)$ and we are done.

If $\alpha$ is a limit ordinal, then we can fix a family of ordinals $(\alpha_n)_{n\in\kappa}$ with $\alpha=\sup_{n\in\kappa}(\alpha_{n}+1)$. By the inductive assumption, for every $n\in\kappa$ there is a subset $A_n\subset H$ such that $A_n\in\tau^{\alpha_n+1}(\F)\setminus\tau^{\alpha_n}(\F)$. After a suitable shift, we can assume that $e\notin A_n$.
Since the ideal $\F$ is $h_n$-invariant, $h_n(A_n)\in\tau^{\alpha_n+1}(\F)\setminus\tau^{\alpha_n}(\F)$ according to Lemma~\ref{hinv}.

Then the set $A=\bigcup_{n\in\w}h_n(A_n)$ does not belong to $\tau^{<\alpha}(\F)$. The inclusion $A\in\tau^{\alpha}(\F)$ will follow as soon as we check that  $A\cap xA\in\tau^{<\alpha}(\F)$ for all $x\in G\setminus \{e\}$. This is clear if $A\cap xA$ is empty. If $A\cap xA$ is not empty, then $x\in h_n(A_n)h_m(A_m)^{-1}\subset H_nH_m$ for some $n,m\in\kappa$. Taking into account that $H_nH_m\cap H_kH_l=\{e\}$ for all $k,l\in\kappa\setminus\{n,m\}$ and $e\notin A$, we conclude that $$A\cap xA\subset h_n(A_n)\cup h_m(A_m)\cup xh_n(A_n)\cup xh_m(A_m)\in\tau^{<\alpha}(\F)$$ as $\tau^{<\alpha}(\F)$ is additive according to Theorem~\ref{ideal}.
\end{proof} 

Let us recall that a family $\F$ of subsets of a group $G$ is called {\em auto-invariant} if for any injective homomorphism $h:G\to G$ a subset $A\subset G$ belongs to $\F$ if and only if $h(A)\in\F$.

\begin{lemma}\label{l5.8} Let $G$ be a free abelian group $G$ and $\F$ be an auto-invariant ideal of subsets of $G$. If $\F$ is not thin-complete, then for each ordinal $\alpha<|G|^+$ the family  $\tau^\alpha(\F)$ is not thin-complete.
\end{lemma}

\begin{proof} Being free abelian, the group $G$ is generated by some linearly independent subset $B\subset G$. Consider the isomorphism $h:G\to 3G$ of $G$ onto the subgroup $3G=\{g^3:g\in G\}$ and observe that $h$ is expanding and for each $z\in B$ we get $z^2\notin 3G$. The ideal $\F$ being auto-invariant, is $h$-invariant. Applying Lemma~\ref{l6.6}, we conclude that $\tau^{\alpha}(\F)\ne\tau^{<\alpha}(\F)$ for all ordinals $\alpha<\w_1$. If the group $G$ is countable, then this is exactly what we need.

Now consider the case of uncountable $\kappa=|G|$. Being free abelian, the group $G$ is isomorphic to the direct sum $\oplus^\kappa\IZ$ of $\kappa$-many copies of the infinite cyclic group $\IZ$. Write the cardinal $\kappa$ as the disjoint union 
$\kappa=\bigcup_{\alpha\in\kappa}\kappa_\alpha$ of $\kappa$ many subsets $\kappa_\alpha\subset\kappa$ of cardinality $|\kappa_\alpha|=\kappa$. For every $\alpha\in\kappa$ consider the free abelian subgroup $G_\alpha=\oplus^{\kappa_\alpha}\IZ$ of $G$ and fix any isomorphism $h_\alpha:G\to G_\alpha$. It is clear that $G_\alpha\oplus G_\beta\cap G_\gamma\oplus G_\delta=\{0\}$ for all ordinals $\alpha,\beta,\gamma,\delta\in\kappa$ with $\{\alpha,\beta\}\cap\{\gamma,\delta\}=\emptyset$.

Being auto-invariant, the ideal $\F$ is $h_\alpha$-invariant for every $\alpha\in \kappa$. Now it is legal to apply Lemma~\ref{l6.7} to conclude that $\tau^{\alpha}(\F)\ne\tau^{<\alpha}(\F)$ for all ordinals $\alpha<\kappa^+$.
\end{proof}

\begin{proof}[Proof of Theorem~\ref{t5.1}] Let $\F$ be a sub-invariant ideal of subsets of a group $G$ and let $H\subset G$ be a free abelian subgroup of cardinality $|H|=|G|$. Assume that $\tau(\F)\cap\mathcal P_H\not\subset\F$. 

Consider the ideal $\F'=\mathcal P_H\cap\F$ of subsets of the group $H$. 
By transfinite induction  it can be shown that $\tau^{\alpha}(\F')=\mathcal P_H\cap\tau^\alpha(\F)$ for all ordinals $\alpha$.

The sub-invariance of $\F$ implies the sub-invariance (and hence auto-invariance) of $\F'$. By Lemma~\ref{l5.8}, we get $\tau^{\alpha}(\F')\ne \tau^{<\alpha}(\F')$ for each $\alpha<|H|^+=|G|^+$. Then also $\tau^*(\F)\ne\tau^{\alpha}(\F)\ne\tau^{<\alpha}(\F)$ for all $\alpha<|G|^+$. 
\end{proof}

\section{The descriptive complexity of the family $\tau^*(\F)$}\label{s6}

In this section given a countable group $G$ and a left-invariant monotone subfamily $\F\subset\mathcal P_G$, we study the descriptive complexity of the family $\tau^*(\F)$, considered as a subspace of the power-set $\mathcal P_G$ endowed with the compact metrizable topology of the Tychonov product $2^G$ (we identify $\mathcal P_G$ with $2^G$ by identifying each subset $A\subset G$ with its characteristic function $\chi_A:G\to 2=\{0,1\}$). 

 \begin{theorem}\label{t7.1} Let $G$ be a countable group and $\F\subset\mathcal P_G$ be a Borel left-invariant lower family of subsets of $G$.
\begin{enumerate}
\item For every ordinal $\alpha<\w_1$ the family $\tau^\alpha(\F)$ is  Borel in $\mathcal P_G$.
\item The family $\tau^*(\F)=\tau^{<\w_1}(\F)$ is coanalytic.
\item If $\tau^*(\F)\ne\tau^\alpha(\F)$ for all $\alpha<\w_1$, then $\tau^*(\F)$ is not Borel in $\mathcal P_G$.
\end{enumerate}
\end{theorem}

\begin{proof} Let us recall that $G_\circ=G\setminus\{e\}$. 

In Section~\ref{s2} to each subset $A\subset G$ we assigned the $\tau$-tree
$$T_A=\{s\in G_\circ^{<\w}:A_s\notin\F\},$$
where for a finite sequence $s=(g_0,\dots,g_{n-1})\in G_\circ^n\subset G_\circ^{<\w}$ we put 
$$A_s=\bigcap_{x_0,\dots,x_{n-1}\in 2^{n}}g_0^{x_0}\cdots g_{n-1}^{x_{n-1}}A.$$

Consider the subspaces $\WF\subset \Tr$ of $\mathcal P_{G_\circ^{<\w}}$, consisting of all (well-founded) lower subtrees of the tree $G_\circ^{<\w}$.
 
\begin{claim} The function $$T_*:\mathcal P_G\to \Tr,\;\;T_{*}:A\mapsto T_A$$is Borel measurable.
\end{claim} 

\begin{proof} The Borel measurability of $T_*$ means that for each open subset $\U\subset\Tr$ the preimage $T_*^{-1}(\U)$ is a  Borel subset of $\mathcal P_G$. Let us observe that the topology of the space $\Tr$ is generated by the sub-base consisting of the sets
$$\mbox{$\langle s\rangle^+=\{T\in\Tr:s\in T\}$ \ and \  
$\langle s\rangle^-=\{T\in\Tr:s\notin T\}$ where $s\in G_\circ^{<\w}$}.$$
Since $\langle s\rangle^-=\Tr\setminus\langle s\rangle^+$, the Borel masurability of $T_*$ will follow as soon as we check that for every $s\in G_\circ^{<\w}$ the preimage $T_*^{-1}(\langle s\rangle^+)=\{A\in\mathcal P_G:s\in T_A\}$ is Borel.

For this observe that the function $$f:\mathcal P_G\times G_\circ^{<\w}\to \mathcal P_G,\;f:(A,s)\mapsto A_s,$$is continuous. Here the tree $G_\circ^{<\w}$ is endowed with the discrete topology.

Since $\F$ is Borel in $\mathcal P_G$, the preimage $\mathcal E=f^{-1}(\mathcal P_G\setminus \F)$ is Borel in $\mathcal P_G\times G_\circ^{<\w}$.
Now observe that for every $s\in G_\circ^{<\w}$ the set
$$T_*^{-1}(\langle s\rangle^+)=\{A\in\mathcal P_G:s\in T_A\}=\{A\in\mathcal P_G:(A,s)\in \mathcal E\}$$ is Borel. 
\end{proof}

By Theorem~\ref{t1.2}, $\tau^*(\F)=T_*^{-1}(\WF)$ and $\tau^\alpha(\F)=T_*^{-1}(\WF_{{-}1{+}\alpha{+}1})$ for $\alpha<\w_1$. Now Theorem~\ref{WF} and the Borel measurablity of the function $T_*$ imply that the preimage $\tau^*(\F)=T_*^{-1}(\WF)$ is coanalytic while $\tau^\alpha(\F)=T_*^{-1}(\WF_{{-}1{+}\alpha{+}1})$ is Borel for every $\alpha<\w_1$, see \cite[14.4]{Ke}.
\smallskip

Now assuming that $\tau^{\alpha+1}(\F)\ne\tau^\alpha(\F)$ for all $\alpha<\w_1$, we shall show that $\tau^*(\F)$ is not Borel. In the opposite case, $\tau^*(\F)$ is analytic and then its image $T_*(\tau^*(\F))\subset\WF$ under the Borel function $T_*$ is an analytic subspace of $\WF$, see \cite[14.4]{Ke}. By Theorem~\ref{WF}(4), $T_*(\tau^*(\F))\subset\WF_{\alpha{+}1}$ for some infinite ordinal 
$\alpha<\w_1$ and thus $\tau^*(\F)=T_*^{-1}(\WF_{\alpha{+}1})=\tau^{\alpha}(\F)$, which is a contradiction.
\end{proof} 
 
Theorems~\ref{t5.1} and \ref{t7.1} imply:

\begin{corollary} For any countable non-torsion group $G$ the ideal $\tau^*(\F_G)\subset\mathcal P_G$ is coanalytic but not analytic.
\end{corollary}

By \cite[26.4]{Ke}, the $\Sigma_1^1$-Determinacy (i.e., the assumption of the determinacy of all analytic games) 
 implies that each coanalytic non-analytic space is $\Pi^1_1$-complete.
By \cite{Ma}, the $\Sigma_1^1$-Determinacy follows from the existence of a measurable cardinal. So, the existence of a measurable cardinal implies that for each countable non-torsion group $G$ the subspace $\tau^*(\F_G)\subset\mathcal P_G$, being coanalytic and non-analytic, is $\Pi^1_1$-complete.

\begin{question} Is the space $\tau^*(\F_\IZ)$ \ $\Pi^1_1$-complete in ZFC? 
\end{question}

\end{document}